\documentclass[11pt]{amsart}
\usepackage[dvipsnames]{color}
\usepackage{amsfonts,amssymb,amsmath,amscd,amstext}
\usepackage{mathtools}
\usepackage[colorlinks=true,linkcolor=teal,citecolor=purple]{hyperref}
\usepackage[utf8]{inputenc}
\usepackage{microtype}
\usepackage{graphicx}
\usepackage{changes}
\usepackage{empheq}
\usepackage[noabbrev,capitalize,nameinlink]{cleveref}
\usepackage{comment}
\usepackage{mathdots}
\usepackage{aligned-overset}
\usepackage[a4paper,top=2.2cm,bottom=2cm,left=2.1cm,right=2cm]{geometry}
\usepackage{newtxtext,newtxmath}
\let\caron\v
\renewcommand{\leq}{\leqslant}
\renewcommand{\geq}{\geqslant}

\renewcommand{\v}{\nu}

\newcommand{\rr}{{\mathbb{R}}}

\newcommand{\Om}{\Omega}
\newcommand{\eps}{\varepsilon}

\definechangesauthor[color=red]{J}
\definecolor{champagne}{rgb}{0.97, 0.91, 0.81}

\definecolor{asparagus}{rgb}{0.53, 0.66, 0.42}

\DeclareMathOperator{\divv}{div}

\newtheorem{theorem}{Theorem}[section]
\newtheorem{proposition}[theorem]{Proposition}
\newtheorem{definition}[theorem]{Definition}

\theoremstyle{definition}

\theoremstyle{remark}

\numberwithin{equation}{section}

\definechangesauthor[color=blue]{JP}
\definechangesauthor[color=purple]{S}

\author[J. Pozuelo]{Juli\'an Pozuelo} 
\address[Juli\'an Pozuelo]{Departamento de Matemáticas, Universidad de Alicante\protect\newline \indent Campus de Sant Vicent del Raspeig s/n, 03690 Sant Vicent del Raspeig, Spain}
\email{julian DOT pozuelo AT ua DOT es}

\author[S.~Verzellesi]{Simone Verzellesi}
\address[Simone Verzellesi]{Dipartimento di Matematica "Tullio Levi-Civita", Università degli Studi di Padova\protect\newline
\indent via Trieste 63, 35131 Padova (PD), Italy}
\email{simone DOT verzellesi AT unipd DOT it}
\title[Serrin problems with vertical boundary behavior]{Serrin problems with vertical boundary behavior}
 
\date{\today}

\thanks{\textit{Acknowledgements}. 
The authors thank Pieralberto Sicbaldi for stimulating conversations about the contents of this paper. 
S. Verzellesi is member of INdAM-GNAMPA, and is supported by the University of Padova and by the INdAM-GNAMPA 2026 Project \emph{Variational, Geometric, and Analytic Perspectives on Regularity}, CUP E53C25002010001. J. Pozuelo is supported by the  grant PID2023-151060NB-I00 funded by MCIN/AEI/10.13039/501100011033, and by ERDF/EU}
\begin{document} 
\begin{abstract}
We study Serrin-type overdetermined problems for a class of possibly degenerate elliptic operators under a vertical boundary condition forcing gradient blow-up. We identify the precise regime in which the problem admits a solution on a ball. In this regime, we prove rigidity: every solution domain is a ball, and the solution is uniquely determined by an explicit radial profile.
\end{abstract}
\maketitle



\section{Introduction and main results}

In a celebrated paper, Serrin \cite{MR333220} showed that the \emph{overdetermined boundary value problem} 
\begin{equation}\label{eq_serrin_problem}
    -\mathcal L u=1\text{ in }\Om,\qquad   u=0\text{ on }\partial\Om,\qquad   \frac{\partial u}{\partial\nu}=c\text{ on }\partial\Om\text{ for some constant }c>0,
\end{equation}
 when
 \begin{equation}\label{eq_laplacian}
     \mathcal{L}u=\Delta u
 \end{equation}
 and $\Om\subseteq\rr^n$ is a bounded domain with smooth boundary and  interior unit normal $\nu$,
is solvable in $C^2(\overline\Om)$ precisely when $\Om$ is a ball. Besides \eqref{eq_laplacian}, the results of \cite{MR333220} actually apply to much more general elliptic operators. Among them, a relevant instance is provided by the \emph{prescribed mean curvature} operator
\begin{equation}\label{eq_intro_pmc}
  \mathcal L u =\divv\left(\frac{\nabla u}{\sqrt{1+|\nabla u|^2}}\right).
\end{equation}
Besides their intrinsic interest, both problems have a strong physical motivation: \eqref{eq_laplacian} arises in the study of the torsion of solid bars and viscous incompressible fluids, and \eqref{eq_intro_pmc} describes capillary surfaces which, due to the Neumann condition in \eqref{eq_serrin_problem}, meet the boundary of a cylindrical container at a constant (non-vanishing) \emph{wetting angle}. Serrin's original proof is based on the \emph{moving planes method}, and has inspired  many subsequent symmetry results \cite{MR544879,MR1487978,MR1463801}. Shortly after \cite{MR333220}, Weinberger \cite{MR333221} provided a different proof of the rigidity of \eqref{eq_serrin_problem}. The core idea of \cite{MR333221} consists in showing the constancy of the \emph{$P$-function}
\begin{equation*}
    P(x)=|\nabla u(x)|^2+\frac{2}{n}u(x)
\end{equation*}
by combining a maximum principle argument with appropriate integral identities (cf. also \cite{MR525971}). These approaches face serious difficulties in the case of \emph{degenerate} elliptic operators. Among the obstructions, (weak) solutions in $C^1(\overline\Om)$ may fail to be $C^2$. A prototypical instance is given by the \emph{$p$-Laplace operator}
\begin{equation}\label{eq_plaplacian}
    \mathcal L u=\divv\left(|\nabla u|^{p-2}\nabla u\right),\qquad 1<p<\infty,
\end{equation}
which is singular when $1<p<2$ and degenerate when $p>2$. The first outcomes in the degenerate setting are available only under appropriate growth conditions on $\mathcal L$ \cite{MR980297}, and may distinguish between singular \cite{MR1776351} and degenerate \cite{MR1947461} regime. A significant improvement has been achieved in \cite{MR2366129,MR2232009} for the fairly general class of operators
\begin{equation}\label{eq_operators_fragala}
    \mathcal L u =\divv\left(A(|\nabla u|)\nabla u\right),
\end{equation}
where $A\in C^2(0,\infty)$ satisfies the (possibly degenerate) ellipticity assumption
\begin{equation}\label{hp_0}\tag{$\mathrm{A}1$}
    \lim_{t\to 0}B(t)=0,\qquad \dot B(t)>0\text{ on }(0,\infty),\qquad\text{}B(t)\coloneq t A(t).
\end{equation}
Clearly, \eqref{eq_laplacian}, \eqref{eq_intro_pmc} and \eqref{eq_plaplacian} fall within this class, which nevertheless allows for operators with more general growth and degeneracy conditions, including for instance the class 
\begin{equation}\label{eq_more_general}
    \mathcal L u=\divv\left(\frac{|\nabla u|^{p-2}}{\left(1+|\nabla u|^2\right)^{q/2}}\nabla u\right),\qquad 1<p<\infty,\qquad 0\leq q\leq p-1.
\end{equation}
 In this paper, we investigate the singular case in which the graph meets the boundary cylinder with a vertical contact angle, forcing $|\nabla u|$ to blow up as $\partial\Om$ is approached. More precisely, we consider the problem
\begingroup
\renewcommand{\theequation}{P}
\begin{subequations}\label{eq_overdetermined_problem}
\renewcommand{\theequation}{\theparentequation\arabic{equation}}

\begin{empheq}[left=\empheqlbrace]{align}
-\divv\left(A(|\nabla u|)\nabla u\right)
    &= 1
    && \text{in } \Om,
\label{eq_overdetermined_problem_pde}
\\
u
    &= 0
    && \text{on } \partial\Om,
\label{eq_overdetermined_problem_dirichlet}
\\
\lim_{x\to x_0}
\frac{\left\langle \nabla u(x),\v(x_0)\right\rangle}
{\sqrt{1+|\nabla u(x)|^2}}
    &= 1
    && \text{for every } x_0\in\partial\Om.
\label{eq_overdetermined_problem_boundary_verticality}
\end{empheq}
\end{subequations}
\endgroup
\addtocounter{equation}{-1}
In the following,  $\Om\subseteq\rr^n$ is a bounded domain with $C^1$ boundary, $\v$ is the interior unit normal to $\partial\Om$ and $A\in C^2(0,\infty)$ satisfies \eqref{hp_0}. Heuristically, the Neumann condition \eqref{eq_overdetermined_problem_boundary_verticality} arises by singularizing the equation
 \begin{equation*}    
\frac{\left\langle \nabla u,\v\right\rangle}
{\sqrt{1+|\nabla u|^2}}=\frac{c}{\sqrt{1+c^2}}\qquad \text{on $\partial\Om$},
 \end{equation*}
 which is satisfied by solutions to \eqref{eq_serrin_problem}. Indeed, \eqref{eq_overdetermined_problem_boundary_verticality}
forces the uniform vertical displacement  \begin{equation}\label{lem_verticality_uniform}
        \lim_{x\to x_0}\frac{\nabla u(x)}
{\sqrt{1+|\nabla u(x)|^2}}=\v(x_0),\qquad    \lim_{x\to x_0}|\nabla u(x)|=+\infty\qquad\text{ uniformly for $x_0\in\partial\Om$.}
    \end{equation}
   In the prototypical case of the prescribed mean curvature operator \eqref{eq_intro_pmc}, \eqref{eq_overdetermined_problem_boundary_verticality} describes the limiting case of \emph{perfectly wetting} capillary surfaces (cf. \cite{MR487722}).
    Because of  \eqref{lem_verticality_uniform}, the lack of $C^1$-boundary regularity of $u$ is forced by the very structure of the problem. Accordingly, solutions to \eqref{eq_overdetermined_problem} are meant as follows.
\begin{definition}\label{defweaksol}
     $u\in C^1(\Om)\cap C^0(\overline\Om)$ is a weak solution to \eqref{eq_overdetermined_problem} if it satisfies \eqref{eq_overdetermined_problem_dirichlet} and \eqref{eq_overdetermined_problem_boundary_verticality}, and moreover.
    \begin{equation}\label{eq_weak_solution_integral_form}
        \int_\Om A(|\nabla u|)\langle\nabla u,\nabla\varphi\rangle\,dx=\int_\Om\varphi\,dx,\qquad\text{  for every $\varphi\in C^\infty_c(\Om)$.}
    \end{equation}
\end{definition}
When \eqref{hp_0} holds, \cite[Proposition]{MR2232009} and \cite[Theorem 1.1]{MR2366129} ensure that \eqref{eq_serrin_problem} is solvable precisely when $\Om$ is a ball, but regardless of further constraints on $A$. On the other hand, the existence of a weak solution to \eqref{eq_overdetermined_problem} forces \emph{a priori} constraints not only on the domain, but also on the differential operator. A first heuristic indication of this phenomenon is that the verticality condition \eqref{eq_overdetermined_problem_boundary_verticality} rules out any operator for which standard elliptic regularity yields $C^1$-boundary regularity for the Dirichlet problem (e.g. the Laplacian), since this would be incompatible with the required blow-up of $|\nabla u|$. More precisely, the following holds.
\begin{theorem}\label{thm_lambda}
    Assume that $u\in C^1(\Om)\cap C^0(\overline\Om)$ is a weak solution to \eqref{eq_overdetermined_problem}. Assume that \eqref{hp_0} holds. Then 
\begin{equation}\label{consequence_1}
        \lambda\coloneq\lim_{t\to\infty} B(t)=\frac{|\Om|}{\mathcal H^{n-1}(\partial\Om)}.
    \end{equation}
    Moreover, $\Om$ is a self-Cheeger set. Namely
    \begin{equation*}
        \frac{\mathcal {H}^{n-1}(\partial\Om)}{|\Om|}=\min\left\{\frac{P(F)}{|F|}\,:\,\text{ $F\subseteq \Om$ has finite perimeter},\,|F|>0\right\}.
    \end{equation*}
\end{theorem}
For instance, when $\mathcal L$ is as in \eqref{eq_more_general}, \eqref{consequence_1} implies that \eqref{eq_overdetermined_problem} is solvable only if $q=p-1$, thus ruling out \eqref{eq_laplacian} and \eqref{eq_plaplacian}. In order to understand the correct regime for the solvability of \eqref{eq_overdetermined_problem}, one may look at the boundary behavior of the relevant $P$-function associated with \eqref{eq_operators_fragala}. According to \cite{MR2366129,MR2232009}, we set  
\begin{equation}\label{def_pfunction}
  P(x)\coloneq \Phi\big(|\nabla u(x)|\big)+\frac{2}{n}u(x),\qquad x\in\Om,
\end{equation} 
where
\begin{equation*}\label{eq_def_phi_e_equiv}
      \Phi(t)\coloneq2\int_0^t s\,\dot B(s)\,ds,\qquad t\in(0,\infty).
\end{equation*}
In the finite-slope setting of \cite{MR2366129,MR2232009}, the boundary continuity of $P$ is crucial to carry out the appropriate maximum-principle argument. Accordingly, we observe that, due to \eqref{eq_overdetermined_problem_dirichlet}, \eqref{eq_overdetermined_problem_boundary_verticality} and \eqref{lem_verticality_uniform}, $P$ extends continuously up to $\partial\Om$ to the constant value $$    \Phi_\infty\coloneq \lim_{t\to\infty}\Phi(t)\in(0,\infty)$$
provided that 
\begin{equation}\label{hp_1}\tag{$\mathrm{A}2$}
        \int_{0}^{\infty}t\,\dot B(t)\,dt<\infty.
    \end{equation}
As we show in the following result, this heuristic consideration indeed identifies the correct class of operators for which \eqref{eq_overdetermined_problem} is solvable in the model case of a spherical domain.
\begin{theorem}\label{thm_radial}
   Problem \eqref{eq_overdetermined_problem} has weak solution $u$ on some ball $B_R$ if and only if \eqref{hp_1} holds.    
    In such cases, \eqref{consequence_1} holds,  $R=n\lambda$, and $u$ is uniquely determined by
    \begin{equation}\label{eq_radial}
        u(x)=\int_{|x|}^{n\lambda} B^{-1}\left(\frac{t}{n}\right)\,dt.
    \end{equation}
\end{theorem}
Once \Cref{thm_radial} is established, the main result of this paper provides the rigidity of \eqref{eq_overdetermined_problem}, which is indeed solvable \emph{precisely} when $\Om$ is a ball.
\begin{theorem}\label{thm_main}
    Assume that there is a weak solution $u\in C^1(\Om)\cap C^0(\overline\Om)$ to \eqref{eq_overdetermined_problem}. Assume that \eqref{hp_0} and \eqref{hp_1} hold. Then, up to translations, $\Om$ is the ball $B_{n\lambda}$, and $u$ is the radially symmetric function defined by \eqref{eq_radial}.
\end{theorem}
We conclude with a brief outline of the proofs. The essential difference from the finite-slope regime lies in the loss of boundary regularity. In the setting of \cite{MR2366129,MR2232009}, the solution is $C^1$ up to $\partial\Omega$, so that the Neumann datum is attained pointwise. Accordingly, a combination of the maximum principle and of an integral identity à la Poho\caron{z}aev-Pucci-Serrin \cite{MR2018671,MR192184,MR855181} can be performed directly on $\Om$, yielding the constancy of the $P$-function. By contrast, \eqref{eq_overdetermined_problem_boundary_verticality} forces $|\nabla u|$ to diverge as $\partial\Omega$ is approached, so that neither $\nabla u$ nor $A(|\nabla u|)\nabla u$ can be handled directly at the boundary. 
To overcome this obstruction, we replace $\partial\Omega$ with a family of regular interior level sets. More precisely, for sufficiently small $\delta>0$, we construct domains $E_\delta\Subset\Omega$ whose boundaries $\Sigma_\delta$ lie in the level set $\{u=\delta\}$. The vertical boundary behavior implies that $\nabla u\neq0$ in a collar neighborhood of $\partial\Omega$, and hence that the hypersurfaces $\Sigma_\delta$ are of class $C^1$. Crucially, although $u$ is not of class $C^1$ up to $\partial\Omega$, \eqref{eq_overdetermined_problem_boundary_verticality} still provides enough control to ensure that 
\[
|E_\delta|\to|\Om|,\qquad \mathcal H^{n-1}(\Sigma_\delta)
\longrightarrow
\mathcal H^{n-1}(\partial\Omega),\qquad \text{as }\delta\to 0.
\]
This convergence plays the role of the boundary regularity available in the finite-slope regime.
The proof of \Cref{thm_lambda} follows by combining the coarea formula with a careful choice of test functions in \eqref{eq_weak_solution_integral_form}, in such a way to exhaust $\Om$. When $\Om$ is a ball, a suitable weak comparison principle yields radial symmetry of the solution, and reduces \eqref{eq_overdetermined_problem_pde} to an ODE. This gives both the explicit radial profile and the sharp solvability condition \eqref{hp_1}, proving \Cref{thm_radial}.
Finally, the proof of \Cref{thm_main} follows the $P$-function strategy, but once again only after localization to the interior sets $E_\delta$. By \eqref{hp_1}, $P$ admits finite constant trace $\Phi_\infty$ at $\partial\Om$. Applying a maximum principle argument on $E_\delta$ and then letting $\delta\to0$, we obtain the upper bound $P\leq\Phi_\infty$. To prove the equality, we combine \eqref{eq_weak_solution_integral_form} with the Poho\caron{z}aev-Pucci-Serrin identity to appropriate vertical shifts of $u$ over $E_\delta$. The aforementioned convergence properties of $E_\delta$ and $\Sigma_\delta$ provide enough control on the remainder terms to pass to the limit and conclude that $P\equiv\Phi_\infty$.
Once the constancy of $P$ is established, the very same rigidity argument as in \cite{MR2366129} forces the level sets of $u$ to be concentric spheres. Consequently, $\Omega$ is a ball, and $u$ coincides with the radial solution described by \eqref{eq_radial}.

\section{Proofs}
\subsection*{Preliminary results}
In the following, $u\in C^1(\Om)\cap C^0(\overline\Om)$ is a fixed weak solution to \eqref{eq_overdetermined_problem}. 
%
%
We denote by $d$ the (inner) signed distance from $\partial\Om$. 
 Since $\partial\Om$ is of class $C^1$, $d$ is Lipschitz continuous. Since $\partial\Om$ is compact, there exists $\bar\eps>0$ such that $x_0+s\v(x_0)\in\Om$ for any $s\in(0,\bar\eps)$.
 We define $\Psi:\partial\Om\times(0,\bar\eps)\to\Om$ by 
\begin{equation*}\label{eq_def_of_change_variables}
    \Psi(x_0,s)=x_0+s\v(x_0).
\end{equation*}
    Although $\Psi$ may not be injective, still $\Psi(\partial\Om\times(0,\bar\eps))=\{0<d<\bar\eps\}$. By \eqref{lem_verticality_uniform}, we choose $\bar\eps$ small so that 
    \begin{equation}\label{eq_grad_along_s}
        \frac{\partial}{\partial s}\Big(u(\Psi(x_0,s))\Big)=\left\langle \nabla u(x_0+s\nu(x_0),\nu(x_0)\right\rangle>1\qquad \text{ on }\partial\Om\times(0,\bar\eps).
    \end{equation}
Therefore, as $u\equiv 0$ on $\partial\Om$, and if $x=\Psi(x_0,s)$ for some $(x_0,s)\in\partial\Om\times(0,\bar\eps)$,
\eqref{eq_grad_along_s} yields
\begin{equation}\label{eq_aux_superlineard}
    u(x)=u(x_0+s\nu(x_0))-u(x_0)\geq s\geq d(x_0+s\nu(x_0))=d(x)\qquad \text{on }\{0<d<\bar\eps\}.
\end{equation}
Fix $\bar\delta<\bar\eps$. For any $0<\delta<\bar\delta$, set 
 \begin{equation}\label{eq_edelta}
     E_\delta=\{u>\delta\}\cup\{d>\bar\delta\}.
 \end{equation}
First, $E_\delta\Subset\Om$ is open and bounded. Moreover, \eqref{eq_aux_superlineard} grants that $
    \{d\geq\bar\delta\}\subseteq E_\delta, $
whence $\partial E_\delta\subseteq \{0<d<\bar\delta\}$. Therefore, since \eqref{eq_grad_along_s} implies that $\nabla u\neq 0$ on $\{0<d<\bar\delta\}$, then
\begin{equation*}
   \Sigma_\delta\coloneq \partial E_\delta=\{u=\delta\}\cap \{0<d<\bar\delta\},
\end{equation*}
and $\Sigma_\delta$ is a hypersurface of class $C^1$.
 In addition, \eqref{eq_aux_superlineard} implies that $\displaystyle{\bigcup_{0<\delta<\bar\delta}
 }E_\delta=\Om$, whence
 \begin{equation}\label{E_delta_conv_lebesgue}
     \lim_{\delta\to 0^+}|E_\delta|=|\Om|.
 \end{equation}
A crucial consequence for the subsequent analysis is the convergence of the measures of $\Sigma_\delta$.
 \begin{proposition}\label{prop_convhaus}
  Let $S=\{(x,u(x))\,:\,x\in\Om\}$. Then $\overline S$ is a hypersurface of class $C^1$. Moreover, 
    \begin{equation}\label{eq_convdellehausdorff}
        \lim_{\delta\to 0^+}\mathcal H^{n-1}(\Sigma_\delta)=\mathcal H^{n-1}(\partial\Om).
    \end{equation}
\end{proposition}
\begin{proof}
The regularity of $\overline S$ may be proved locally near $\partial\Om$. Fix $\hat x \in\partial\Om$. As \Cref{defweaksol} is invariant under isometries, we assume that $\hat x=0$ and $\v(\hat x)=-\frac{\partial}{\partial x_1}$. We denote points in $\rr^n$ by $(x_1,x')$. Since $\partial\Om$ is of class $C^{1}$, there exist a neighborhood $U'$ of $0\in\rr^{n-1}$, a neighborhood $U_1$ of $0\in\rr$ and $\omega\in C^{1}(\overline{U'})$ such that 
    \begin{equation*}
        \Om\cap\left( U_1\times U'\right)=\{(x_1,x')\in\rr^n\,:\,x_1<\omega(x')\},\qquad  \partial\Om\cap\left( U_1\times U'\right)=\{(x_1,x')\in\rr^n\,:\,x_1=\omega(x')\}.
    \end{equation*}
    In particular, denoting by $\nabla'$ the gradient with respect to $x'$, 
      \begin{equation*}
        \nu\big((\omega(x'),x')\big)=-\frac{1}{\sqrt{1+|\nabla'\omega(x')|^2}}\frac{\partial}{\partial x_1}+\sum_{i=2}^n \frac{\partial_{i}\omega (x')}{\sqrt{1+|\nabla'\omega(x')|^2}}\frac{\partial}{\partial x_i}.
    \end{equation*}
    Therefore, by \eqref{lem_verticality_uniform} and for any $\tilde x\in \partial\Om\cap\left( U_1\times U'\right)$, 
    \begin{equation}\label{aux_limiteverticalitàderivatetang}
        \lim_{x\to \tilde x}\frac{\partial _{1}u}{|\nabla u|}=-\frac{1}{\sqrt{1+|\nabla'\omega(\tilde x')|^2}},\qquad  \lim_{x\to \tilde x}\frac{\partial _{i}u}{|\nabla u|}=\frac{\partial_{i}\omega (\tilde x')}{\sqrt{1+|\nabla'\omega(\tilde x')|^2}}\quad\text{ for }i=2,\ldots,n.
    \end{equation}
    Up to smaller neighborhoods, \eqref{eq_overdetermined_problem_boundary_verticality} implies  that 
    \begin{equation}\label{aux_proprietadiuxuno}
        \partial _{1}u<0\text{ on }(U_1\times U')\cap\Om,\qquad\lim_{x\to x_0}\partial _{1}u=-\infty\text{ for any }x_0\in(U_1\times U')\cap \partial\Om.
    \end{equation}
    In particular, $u>0\text{ on }(U_1\times U')\cap\Om$.  Therefore, up to smaller neighborhoods, the implicit function theorem yields the existence of $\gamma>0$ and $\psi\in C^1(U'\times(0,\gamma))$ such that 
    \begin{equation}\label{aux_dini_1}
        u\left(\psi(x',x_{n+1}),x'\right)=x_{n+1}\qquad\text{ on } U'\times(0,\gamma).
    \end{equation}
    Therefore, $S\cap \left(U_1\times U'\times(0,\gamma)\right)$ is the $x_{1}$-graph of $\psi$. Notice that $\psi$ can be extended with continuity up to $U'\times\{0\}$ by setting $\psi(x',0)=\omega(x')$. Indeed, assume that $((x_k)',(x_k)_{n+1})\to (x',0)$ as $k\to\infty$. Chose an arbitrary subsequence, without relabeling it. Notice that $\left\{\psi\big((x_k)',(x_k)_{n+1}\big)\right\}_{k}$ is bounded: up to a further subsequence, it converges to some $\tilde \psi\in\rr$. Passing to the limit in \eqref{aux_dini_1}, we infer that $u(\tilde \psi,x')=0$. Since $u>0$ locally by \eqref{eq_aux_superlineard}, the unique possibility is that $\tilde \psi=\omega(x')$. Therefore $\psi$ is continuous up to $U'\times\{0\}$.
   We show that $\psi$ is of class $C^1$ up to $U'\times\{0\}$. Differentiating \eqref{aux_dini_1} at $(x',x_{n+1})\in U'\times(0,\gamma)$ yields
    \begin{equation}\label{aux_derivateprimeuv}
        \partial_{i}\psi(x',x_{n+1})=-\frac{\partial_{i}u\left(\psi(x',x_{n+1}),x'\right)}{ \partial_{1}u\left(\psi(x',x_{n+1}),x'\right)}\quad i=2,\ldots, n,\quad  \partial_{n+1}\psi(x',x_{n+1})=\frac{1}{ \partial_{1}u\left(\psi(x',x_{n+1}),x'\right)}
    \end{equation}
   Therefore, for $i=2,\ldots,n$, for any fixed $\tilde x'\in U'$ and by the continuity of $\psi$,
    \begin{equation}\label{eq_convergenza_derivate_tang}
        \lim_{(x',x_{n+1})\to(\tilde x',0)}\partial_{i}\psi(x',x_{n+1})\overset{\eqref{aux_derivateprimeuv}}{=} \lim_{(x',x_{n+1})\to(\tilde x',0)}-\frac{\partial_{i}u\left(\psi(x',x_{n+1}),x'\right)}{ \partial_{1}u\left(\psi(x',x_{n+1}),x'\right)}\overset{\eqref{aux_limiteverticalitàderivatetang}}{=}\partial _{i}\omega(\tilde x')=\partial _{i}\psi(\tilde x',0).
    \end{equation}
    Therefore, the tangential (with respect to $U'\times\{0\}$)  derivatives of $\psi$ are continuous up to $U'\times\{0\}$. Moreover,
    \begin{equation}\label{aux_derinorm1}
        \lim_{(x',x_{n+1})\to(\tilde x',0)}\partial_{n+1}\psi(x',x_{n+1})\overset{\eqref{aux_derivateprimeuv}}{=}  \lim_{(x',x_{n+1})\to(\tilde x',0)}\frac{1}{ \partial_{1}u\left(\psi(x',x_{n+1}),x'\right)}\overset{\eqref{aux_proprietadiuxuno}}{=}0.
    \end{equation}
By Lagrange mean value theorem, there exists $\psi(\tilde x',x_{n+1})<\xi(\tilde x',x_{n+1})<\psi(\tilde x',0)=\omega(x')$ such that 
    \begin{equation}\label{aux_lagrange}
       x_{n+1}\overset{\eqref{aux_dini_1}}{=}u\big(\psi(\tilde x',x_{n+1}),\tilde x'\big)-u\big(\psi(\tilde x',0),\tilde x'\big)=\partial _{1} u\big(\xi(\tilde x',x_{n+1}),\tilde x'\big)\big(\psi(\tilde x',x_{n+1})-\psi(\tilde x',0)\big).
    \end{equation}
    Therefore, noticing that $\xi(\tilde x',x_{n+1})\to \omega(x')$ as $x_{n+1}\to 0^+,$
    \begin{equation}\label{aux_derinorm2}
        \partial^+_{n+1} \psi(\tilde x',0)=\lim_{x_{n+1}\to 0^+}\frac{\psi(\tilde x,x_{n+1})-\psi(\tilde x,0)}{x_{n+1}}\overset{\eqref{aux_lagrange}}{=}\lim_{x_{n+1}\to 0^+}\frac{1}{\partial _{1} u\big(\xi(\tilde x',x_{n+1}),\tilde x'\big)}\overset{\eqref{aux_proprietadiuxuno}}{=}0.
    \end{equation}
    By \eqref{aux_derinorm1} and \eqref{aux_derinorm2}, the normal derivative of $v$ is continuous up to $U'\times\{0\}$. To prove \eqref{eq_convdellehausdorff}, notice that
    \begin{equation*}
        \Sigma_\delta\cap \left(U_1\times U'\right)=\{(\psi(x',\delta),x')\,:\,x'\in U'\},\qquad \partial\Om\cap \left(U_1\times U'\right)=\{(\omega(x'),x')\,:\,x'\in U'\}.
    \end{equation*}
    Therefore, for an arbitrary $\varphi\in C^\infty_c(U_1\times U') $, the area formula yields
    \begin{equation}\label{eq_area_formula_hnmenouno}
    \begin{split}
           \lim_{\delta\to 0}\int_{\Sigma_\delta\cap \left(U_1\times U'\right)}\varphi\,d\mathcal H^{n-1}&=\lim_{\delta\to 0}\int_{U'}\varphi(\psi(x',\delta),x')\sqrt{1+|\nabla' \psi(x',\delta)|^2}\,dx'\\
           &\overset{\eqref{eq_convergenza_derivate_tang}}{=}\int_{U'}\varphi(\omega(x'),x')\sqrt{1+|\nabla' \omega(x')|^2}\,dx'\\
           &=\int_{\partial\Om\cap \left(U_1\times U'\right)}\varphi\,d\mathcal H^{n-1}.
    \end{split}    
    \end{equation}
    The choice of a finite partition of unity in a neighborhood of $\partial\Om$ and \eqref{eq_area_formula_hnmenouno} prove \eqref{eq_convdellehausdorff}.
 \end{proof}
 We conclude this section with the following (weak) comparison principle. \begin{proposition}\label{prop_comp_princ}
     Let $v,w\in C^1(\Om)\cap C^0(\overline\Om)$ satisfy \eqref{eq_weak_solution_integral_form}. If $v\leq w$ on $\partial\Om$, then $v\leq w$ on $\overline\Om$.
 \end{proposition}
 \begin{proof}
     Choose any family $(\Om_\eps)_{\eps}$ of bounded domains, each of them compactly contained in $\Om$ and with $C^1$ boundary,  with the property that any converging sequence $(x_\eps)_\eps$ with $x_\eps\in\partial\Om_\eps$ converges to a point in $\partial\Om$. Set $M_\eps=\max_{\partial\Om_\eps}(v-w)$ and $w_\eps=w+M_\eps$. Then $w_\eps$ satisfies \eqref{eq_weak_solution_integral_form}. By construction, $v,w_\eps\in C^1(\overline \Om_\eps)$ and $v\leq w_\eps$ on $\partial\Om_\eps$. Then, by \cite[Lemma 3.7]{MR2232009}, $v\leq w+M_\eps$ on $\overline{\Om}_\eps.$ Since $v\leq w$ on $\partial\Om$ and both are continuous over $\overline\Om$, the choice of $(\Om_\eps)_\eps$ yields $\displaystyle{\lim_{\eps\to 0}}M_\eps\leq 0$. Therefore, $v\leq w$ on $\overline\Om$.
 \end{proof}
\subsection*{Proof of \Cref{thm_lambda}}
    By \eqref{hp_0}, $B\geq 0$, and it admits a limit $\lambda\in [0,\infty]$ as $t\to\infty$. Since $\partial\Om$ is $C^1$, there is an open neighborhood $U\subseteq\overline\Om\cap\{0\leq d\leq \bar\eps\}$ of $\partial\Om$, and $\varrho\in C^1(U)$, such that $\partial\Om=\{\varrho=0\}$, $\varrho>0$ on $U\cap\Om$ and $\nabla\varrho\neq 0$ on $U$. Up to a smaller neighborhood, there exist $\eps_1,\eps_2>0$ such that $\varrho<\eps_1$ on $U$, $|\nabla\varrho|>\eps_2$ on $U$ and $\varrho=\eps_1$ on $\partial U\cap\Om$ . In particular, $\frac{\nabla\varrho}{|\nabla\varrho|}=\nu$ on $\partial\Om$.  By \eqref{lem_verticality_uniform}, we assume that $|\nabla u|\neq 0$ on $U$.  Setting $\Gamma_\eps=\{\varrho=\eps\}$, one argues \emph{verbatim} as in the proof of \Cref{prop_convhaus} and as in \eqref{eq_aux_superlineard} to infer
    \begin{equation}\label{eq_convdellehausdorffgamma}
        \lim_{\eps\to 0}\mathcal H^{n-1}(\Gamma_\eps)=\mathcal H^{n-1}(\partial\Om),\qquad \eps_2 d(x)\leq \eps\text{ on }\Gamma_\eps.
    \end{equation}
    Fix $\eta\in C^\infty\big([0,\infty)\big)$ such that $\eta\equiv 0$ on $\left[0,\frac{1}{3}\right]$, $\eta\equiv 1$ on $\left[\frac{2}{3},\infty\right)$ and $\dot \eta\geq 0$.
    For $0<\eps<\eps_1,$ define
    \begin{equation*}
        \varphi_\eps(x)=
\begin{cases}
\displaystyle{\eta\left(\frac{\varrho(x)}{\eps}\right)} & \text{if } x\in U,\\
1  & \text{if } x\in\Om\setminus U.
\end{cases}
    \end{equation*}
  Then $\varphi_\eps\in C^1_c(\Om)$. Therefore, \eqref{eq_weak_solution_integral_form} and the definition of $\varphi_\eps$ imply that 
    \begin{equation}\label{aux_proof_limit_tat_1}
        \int_\Om \varphi_\eps\,dx
        =\frac{1}{\eps}\int_{\{0<\varrho<\eps\}}B(|\nabla u|)\left\langle\frac{\nabla u}{|\nabla u|},\frac{\nabla\varrho}{|\nabla\varrho|}\right\rangle\eta'\left(\frac{\varrho}{\eps}\right)|\nabla\varrho|\,dx
    \end{equation}
    On the one hand, the left hand side of \eqref{aux_proof_limit_tat_1} converges to $|\Om|$ as $\eps\to 0$. On the other hand, by the coarea formula and a change of variables, the right hand side of \eqref{aux_proof_limit_tat_1} reads as 
    \begin{equation}\label{aux_proof_limit_tat_2}
      \int_0^1 \eta'\left(\tau\right)\int_{\Gamma_{\tau\eps}}B(|\nabla u|)\left\langle\frac{\nabla u}{|\nabla u|},\frac{\nabla\varrho}{|\nabla\varrho|}\right\rangle \,d\mathcal H^{n-1}\,d\tau.
    \end{equation}
    We claim that $\lambda$ is finite. Assume not by contradiction. By \eqref{aux_proof_limit_tat_1},  \eqref{aux_proof_limit_tat_2} is bounded above by $|\Om|$. Therefore, by  \eqref{lem_verticality_uniform} and \eqref{eq_convdellehausdorffgamma}, and since $\frac{\nabla\varrho}{|\nabla\varrho|}=\nu$ on $\partial\Om$, for any $M>0$ there is $\eps_M>0$ such that 
    \begin{equation*}
        |\Om|\geq\int_0^1 \eta'\left(\tau\right)\int_{\Gamma_{\tau\eps}}B(|\nabla u|)\left\langle\frac{\nabla u}{|\nabla u|},\frac{\nabla\varrho}{|\nabla\varrho|}\right\rangle \,d\mathcal H^{n-1}\,d\tau\geq M\int_0^1\eta'\left(\tau\right)\,d\tau=M
    \end{equation*}
    for any $0<\eps<\eps_M$. Since $M>0$ is arbitrary, a contradiction follows. Hence $\lambda$ is finite. Therefore, by the dominated convergence theorem, \eqref{lem_verticality_uniform} and \eqref{eq_convdellehausdorffgamma} imply that  \eqref{aux_proof_limit_tat_2} converges to $\lambda\mathcal H^{n-1}(\partial\Om)$ as $\eps\to 0$. This concludes the proof of \eqref{consequence_1}. 
  To conclude, notice that \eqref{consequence_1} implies that $x\mapsto A(|\nabla u(x)|)\nabla u(x)$ is a \emph{bounded divergence-measure vector field} on $\Om$, e.g. in the sense of \cite{MR2118477}. Therefore, \eqref{consequence_1} and the weak divergence theorem proved in \cite[Theorem 2.9]{MR4819842} ensures that $\Om$ is self-Cheeger.
\subsection*{Proof of \Cref{thm_radial}}
Assume that \eqref{eq_overdetermined_problem} has a weak solution $u$ over a ball $B_R$. Then $u$ is radially symmetric. Otherwise, there would exist a rotation $T$ such that $u\circ T\neq u$ over $B_R$. Since \Cref{defweaksol} is invariant under isometries, $u\circ T$ solves \eqref{eq_overdetermined_problem} over $T^{-1}(B_R)=B_R$. But then, by \Cref{prop_comp_princ}, $u=u\circ T$, a contradiction.  Then $u(x)=g(|x|)$ for some profile $g$. By the proof of \cite[Proposition 2.2]{MR2232009}, $\dot g<0$ on $(0,R)$, and 
\begin{equation*}
    B(-\dot g(r))=\frac{r}{n}\qquad\text{on }(0,R).
\end{equation*}
Since $B(0,\infty)=(0,\lambda)$ by \eqref{hp_0} and \eqref{consequence_1}, then $R\leq n\lambda$. Moreover, $B$ is invertible by \eqref{hp_0}, whence
\begin{equation*}
   \dot g(r)=-B^{-1}\left(\frac{r}{n}\right)\qquad\text{on }(0,R).
\end{equation*}
By \eqref{eq_overdetermined_problem_boundary_verticality},  $|\nabla u(x)|=\left|B^{-1}(|x|/n)\right|$ blows up at $\partial B_R$, 
hence $R=n\lambda$. Therefore
\begin{equation}\label{eq_aux_tfcinradial}
    g(r)=g(0)-\int_0^r B^{-1}\left(\frac{\tau}{n}\right)\,d\tau\qquad\text{on }(0,n\lambda).
\end{equation}
By \eqref{eq_overdetermined_problem_dirichlet}, $\displaystyle{\lim_{r\to n\lambda}}g(r)=0$. Therefore, by \eqref{eq_aux_tfcinradial},
\begin{equation*}
    g(0)=\lim_{r\to n\lambda}\int_0^r B^{-1}\left(\frac{\tau}{n}\right)\,d\tau=\lim_{r\to n\lambda}n\int_0^{B^{-1}\left(\frac{r}{n}\right)}t\,\dot B(t)\,dt=n\int_0^{\infty}t\,\dot B(t)\,dt,
\end{equation*}
whence \eqref{hp_1} holds. Conversely, if \eqref{hp_1} holds, the above computation implies that $u$ as in \eqref{eq_radial} solves \eqref{eq_overdetermined_problem}.
\subsection*{Proof of \Cref{thm_main}} By \eqref{hp_1}, \eqref{eq_overdetermined_problem_dirichlet} and \eqref{eq_overdetermined_problem_boundary_verticality}, the $P$-function $P$ defined in \eqref{def_pfunction}
extend continuously to $\partial\Om$ by setting $P(x_0)= \Phi_\infty$ for any $x_0\in\partial\Om$. We claim that
\begin{equation}\label{eq_Pconstant}
    P\equiv \Phi_\infty\qquad\text{on }\overline\Om.
\end{equation}
To this aim, fix $\delta\in(0,\bar\delta)$, and let $E_\delta$ be as in \eqref{eq_edelta}. Since $u\in C^1(\Om)$ and $E_\delta\Subset\Om$, then $\nabla u\in C^0(\overline{E}_\delta,\rr^n)$. Therefore, although $P$ is not necessarily constant over $\partial E_\delta=\Sigma_\delta$, the maximum principle argument as in the proof of \cite[Lemma 3.2]{MR2232009} ensures that 
\begin{equation}\label{stimadasopradelta}
    P(x)\leq \max_{\Sigma_\delta}P= \max_{\Sigma_\delta}\Phi(|\nabla u|)+\frac{2\delta}{n}\qquad\text{for every $x\in E_\delta$}.
\end{equation}
Let $x_\delta\in\Sigma_\delta$ be such that the maximum in the right hand side of \eqref{stimadasopradelta} is assumed at $x_\delta$. By \eqref{eq_aux_superlineard}, $d(x_\delta)\leq u(x_\delta)=\delta$, so that $\displaystyle{\lim_{\delta\to 0}}\,d(x_\delta)=0.$ In particular, letting $\delta \to 0$ in \eqref{stimadasopradelta},
\begin{equation*}
    P(x)\leq\lim_{\delta\to 0}\Phi(|\nabla u(x_\delta)|)\overset{\eqref{lem_verticality_uniform}}{=}\Phi_\infty\qquad\text{for every $x\in \Om$}.
\end{equation*}
The above inequality implies \eqref{eq_Pconstant} provided that
\begin{equation}\label{eq_integral_identity}
    \int_\Om P(x)\,dx=\Phi_\infty |\Om|.
\end{equation}
Fix $0<\delta<\bar\delta$. Define $u_\delta\coloneq u-\delta$ on $\overline{E}_\delta$. Then $u_\delta\in C^1(\overline{E}_\delta)$ and $u_\delta\equiv 0$ on $\partial E_\delta$. Moreover, since $E_\delta\Subset\Om$, $u_\delta$ satisfies \eqref{eq_weak_solution_integral_form} on $E_\delta$. A standard approximation argument grants that $u_\delta$ satisfies \eqref{eq_weak_solution_integral_form} for every $\varphi\in C^1_0(\overline{E}_\delta)$. In particular, since $u_\delta\in C^1_0(\overline{E}_\delta)$ and $\nabla u_\delta=\nabla u$,
\begin{equation}\label{eq_test_udelta}
      \int_{E_\delta} B(|\nabla u|)|\nabla u|\,dx=\int_{E_\delta}u_\delta\,dx.
\end{equation}
On the other hand, notice that $ A(|\nabla u|)\nabla u=\nabla\mathcal L(\nabla u)$, where
\begin{equation*}
  \mathcal L(p)\coloneq \int_0^{|p|}B(\tau)\,d\tau
\end{equation*}
is strictly convex in view of \eqref{hp_0}. Therefore, applying \cite[Theorem 2]{MR2018671} with $a(x)\equiv0$ and $h(x)\equiv x$ yields
\begin{equation*}
    n\int_{E_\delta}\mathcal L(\nabla u)\,dx-\int_{E_\delta}B(|\nabla u|)|\nabla u|\,dx+\int_{E_\delta}\langle x,\nabla u\rangle\,dx=\int_{\Sigma_\delta}\big(\mathcal L(\nabla u)-B(|\nabla u|)|\nabla u|\big)\left\langle x,\nu^e_\delta\right\rangle\,d\mathcal H^{n-1},
\end{equation*}
where $\v^e_\delta$ is the outer unit normal to $E_\delta$. Noticing that 
\begin{equation*}
    \Phi(|\nabla u|)=2|\nabla u| B(|\nabla u|)-2\mathcal L (\nabla u),
\end{equation*}
and since $u_\delta=0$ on $\Sigma_\delta$, the divergence theorem yields
\begin{equation*}
    -\frac{n}{2}\int_{E_\delta}\Phi(|\nabla u|)\,dx+(n-1)\int_{E_\delta}B(|\nabla u|)|\nabla u|\,dx-n\int_{E_\delta} u_\delta\,dx=-\frac{1}{2}\int_{\Sigma_\delta}\Phi\left(|\nabla u|\right)\left\langle x,\nu^e_\delta\right\rangle\,d\mathcal H^{n-1}.
\end{equation*}
Recalling the definition of $P$, exploiting \eqref{eq_test_udelta} and multiplying by $-\frac{2}{n}$,
\begin{equation}\label{eq_aux_pog_pre}
  \int_{E_\delta}P(x)\,dx-\frac{2\delta}{n}|E_\delta|=\frac{1}{n}\int_{\Sigma_\delta}\Phi\left(|\nabla u|\right)\left\langle x,\nu^e_\delta\right\rangle\,d\mathcal H^{n-1}.
\end{equation}
Since $P\in C^0(\overline\Om)$, \eqref{E_delta_conv_lebesgue} and the dominated convergence theorem imply that the left hand side of \eqref{eq_aux_pog_pre} converges to the left hand side of \eqref{eq_integral_identity} as $\delta\to 0$. On the other hand, the divergence theorem grants that
\begin{equation}\label{eq_div_thm_boundary_delta}
    \frac{1}{n}\int_{\Sigma_\delta}\Phi\left(|\nabla u|\right)\left\langle x,\nu^e_\delta\right\rangle\,d\mathcal H^{n-1}=\frac{1}{n}\int_{\Sigma_\delta}\big(\Phi\left(|\nabla u|\right)-\Phi_\infty\big)\left\langle x,\nu^e_\delta\right\rangle\,d\mathcal H^{n-1}+\Phi_\infty|E_\delta|.
\end{equation}
Set $M=\max\{|x|\,:\,x\in\overline\Om\}$. Then
\begin{equation}\label{eq_stimaconhaus}
  \frac{1}{n}\left|\int_{\Sigma_\delta}\big(\Phi\left(|\nabla u|\right)-\Phi_\infty\big)\left\langle x,\nu^e_\delta\right\rangle\,d\mathcal H^{n-1}\right|\leq M\,\mathcal H^{n-1}(\Sigma_{\delta})\,\max_{\Sigma_\delta}|\Phi-\Phi_\infty|.
\end{equation}
Combining \eqref{lem_verticality_uniform} with \eqref{eq_convdellehausdorff}, \eqref{E_delta_conv_lebesgue} and with the boundary continuity of $\Phi$, \eqref{eq_stimaconhaus} and \eqref{eq_div_thm_boundary_delta} yield
\begin{equation*}
    \lim_{\delta\to 0} \frac{1}{n}\int_{\Sigma_\delta}\Phi\left(|\nabla u|\right)\left\langle x,\nu^e_\delta\right\rangle\,d\mathcal H^{n-1}=\Phi_\infty|\Om|,
\end{equation*}
whence \eqref{eq_integral_identity} follows. By the above considerations, \eqref{eq_Pconstant} holds. In order to conclude the proof, it suffices to observe that the second step of the proof of \cite[Theorem 1.1]{MR2366129} works regardless of the boundary regularity of $u$: \eqref{eq_Pconstant} implies that the level sets $\{u=\delta\}$, for $\delta\in(0,\max u)$, are concentric spheres. This fact grants that $\Om$ is a ball, and that $u$ is radial. The rest of the thesis then follows by \Cref{thm_radial}.
\bibliographystyle{abbrv}
\bibliography{biblio}
\end{document}